\documentclass[11pt,a4paper]{article}

\usepackage{amssymb}
\usepackage{amsmath}
\usepackage{amsthm}
\usepackage[normalem]{ulem} 	
\usepackage{xcolor,cancel}

\usepackage[T1]{fontenc}
\usepackage[latin1]{inputenc}
\usepackage{color}

\usepackage{amsfonts}
\usepackage{amssymb}

\usepackage{empheq}

\usepackage[T1]{fontenc}   
\usepackage{lmodern}       
\usepackage{microtype}     
\usepackage{hyperref}      	
\usepackage{verbatim}	   	

\usepackage{geometry}
\usepackage{graphicx}
\usepackage{subfigure}
\usepackage{tikz}
\usepackage{xcolor}
\usepackage{placeins}
\usepackage{wrapfig}
\usepackage{caption}
\usepackage{setspace}
\usepackage[capposition=top]{floatrow}

\hypersetup{
    colorlinks=true,
    linkcolor=blue,
    citecolor=blue,      
    urlcolor=blue
}

\swapnumbers
\numberwithin{equation}{section}

\theoremstyle{plain}
\newtheorem{theorem}{Theorem}[section]
\newtheorem{lemma}[theorem]{Lemma}
\newtheorem{corollary}[theorem]{Corollary}

\theoremstyle{definition}
\newtheorem{problem}[theorem]{Problem}
\newtheorem{remark}[theorem]{Remark}
\newtheorem*{remark*}{Remark}

\DeclareMathOperator{\area}{area}
\DeclareMathOperator{\vol}{vol}
\DeclareMathOperator{\iso}{iso}

\DeclareMathOperator{\im}{image}
\DeclareMathOperator{\tr}{tr}
\DeclareMathOperator{\genus}{genus}

\DeclareMathOperator{\dist}{distance}

\title{A strict inequality for the minimisation of the Willmore functional under isoperimetric constraint}
\author{Andrea Mondino\thanks{University of Oxford. Mathematical Institute. (UK). Email: Andrea.Mondino@maths.ox.ac.uk} \and Christian Scharrer\thanks{University of Warwick. Mathematics Institute. Coventry (UK). Email: C.Scharrer@warwick.ac.uk}}

\begin{document} 
\maketitle

\begin{abstract}
Inspired by previous work of Kusner and Bauer-Kuwert, we prove a strict inequality between the Willmore energies of two surfaces and their connected sum in the context of isoperimetric constraints. Building on previous work by Keller-Mondino-Rivi\`ere, our strict inequality leads to existence of minimisers for the isoperimetric constrained Willmore problem in every genus, provided the minimal energy lies strictly below $8\pi$. Besides the geometric interest, such a minimisation problem has been studied in the literature as a simplified model in the theory of lipid bilayer cell membranes.
\end{abstract}

\section{Introduction}
\label{sec:intro}
The Willmore energy of an immersed closed surface $f: \Sigma \to \mathbb R^3$ is given by
\begin{equation*}
    \mathcal W(f) = \frac{1}{4}\int_{\Sigma} H^2 \, \mathrm d \mu\;,
\end{equation*}
where the mean curvature $H$ is defined as the sum of the principal curvatures, and $\mu$ is the Radon measure corresponding to the pull back metric of the Euclidean metric along~$f$. The isoperimetric ratio is defined by
\begin{equation}\label{eq:defiso_intro}
    \iso(f) = \frac{\area(f)}{\vol(f)^\frac{2}{3}}\;,
\end{equation}
where 
\begin{equation} \label{eq:intro:def_volume}
    \area(f) = \int_{\Sigma}1\,\mathrm d\mu, \qquad \vol(f) = \frac{1}{3}\int_{\Sigma} n \cdot f \,\mathrm d \mu
\end{equation}
are the area and enclosed volume, and $n: \Sigma \to \mathbb S^2$ is the Gau{\ss} map. One can find different definitions of isoperimetric ratio in the literature. Note that with the choice \eqref{eq:defiso_intro}, $\iso(f)$ is invariant under constant scaling of $f$, and it is minimised by any parametrization of the round sphere ${\mathbb S}^2\subset \mathbb R^3$, as a consequence of the Euclidean isoperimetric inequality. Thus
\begin{equation*}
	\im(\iso) = [\sqrt[3]{36\pi},\infty)\; , 
\end{equation*}
where the image is taken over the class of immersed surfaces. Denote with $\mathcal S_g$ the set of smooth immersions $f:\Sigma \to \mathbb R^3$ where $\Sigma$ is a closed surface (i.e. compact without boundary) with $\genus(\Sigma) = g$. We are interested in the following minimisation problem.
\begin{problem}[Isoperimetric constrained Willmore problem] \label{prob:constrained_Willmore}
	Let $g$ be a non-negative integer, and fix $\sigma > \sqrt[3]{36\pi}$. Minimise the Willmore energy $\mathcal W$ in the class of immersions $f \in \mathcal S_g$ subject to the constraint $\iso(f) = \sigma$. 
	That is, find $f_0 \in \{f\in\mathcal S_g: \iso(f)=\sigma\}$ such that 
	\begin{equation}\label{eq:intro:constrained_minimiser}
		\mathcal W(f_0) \leq \mathcal W(f), \quad \text{for any $f\in \mathcal S_g$ with $\iso(f) = \sigma$.}
	\end{equation} 
\end{problem}
Such an immersion $f_0$ satisfying~\eqref{eq:intro:constrained_minimiser} is referred to as \emph{solution} or \emph{minimiser}. 

Beyond the geometric interest, the minimisation problem \ref{prob:constrained_Willmore} is partially motivated by a model for closed lipid bilayer cell membranes  proposed by \textsc{Canham}~\cite{canham1970minimum} and  \textsc{Helfrich}~\cite{helfrich1973elastic}. Indeed, the Willmore energy is the main term in the Canham-Helfrich functional which describes the free energy of a closed lipid bilayer:
$$
F_{\text{Can-Hel}}:=\int_{\text{lipid bilayer}} \left( \frac{k_{c}}{2} (2 H+c_{0})^{2} +\bar{k} K +\lambda \right) +p\cdot V\; ,
$$
where $c_{0}$ is the spontaneous curvature, $k_{c}$ and $\bar{k}$ are bending rigidities, $\lambda$ is the surface tension, $K$ is the Gauss curvature, $p$ is the osmotic pressure and $V$ the enclosed volume.
According to such a model, the shapes of cell membranes observed in nature correspond to  (local) minimisers of $F_{\text{Can-Hel}}$. Notice that, if $c_{0}=\lambda=p=0$ one obtains the Willmore functional (up to a scaling factor and a topological term, by Gauss-Bonnet theorem). If instead $\lambda$ and $p$ do not vanish, they can be seen as Lagrange multipliers for area and volume constraints. Thanks to the scaling invariance of the Willmore functional, such a constrained problem is thus strictly related to the isoperimetric constrained Willmore problem \ref{prob:constrained_Willmore}. 

Even if spherical membranes are most common, also higher genus membranes have been observed in nature: for toroidal shapes see \cite{MB91, MB95a} and for higher genus  see \cite{MBF94, MB95, SL95}.  The Canham-Helfrich and Willmore energies are commonly used in mathematical biology, for instance in modelling red blood cells \cite{canham1970minimum,Mil11}, crista junctions in mitochondria \cite{Ren02}, folds of endoplasmatic reticulum \cite{Shi10}. In particular, the isoperimetric constrained Willmore problem was studied in the axially symmetric case by numerical approximation of the corresponding ordinary differential equations in \cite{seifert1991shape} (see also \cite{L91}). Without symmetry assumptions, the existence of spherical minimisers was achieved  by  \textsc{Schygulla} \cite{MR2928137} and the higher genus case was investigated by \textsc{Keller-Mondino-Rivi\`ere} \cite{MR3176354}. The goal of the paper is to establish the following  (for more details see Corollary \ref{cor:ExMinICW}, later in the introduction):

\begin{theorem}\label{thm:ExMinICW}
Let $g$ be a non-negative integer, and fix  $\sigma > \sqrt[3]{36\pi}$. 
Assume that
\begin{equation} \label{eq:intro:min_iso_constr_energy} 
	 \beta_g(\sigma) := \inf\{\mathcal W(f): f\in \mathcal S_g,\,\iso(f) = \sigma\}<8\pi.
\end{equation}
Then $\beta_g(\sigma)$ is attained by a smoothly embedded minimiser $f_0\in  \mathcal S_g$, i.e. $f_0$ satisfies \eqref{eq:intro:constrained_minimiser}.
\end{theorem}

In the remaining part of the introduction, we set Theorem \ref{thm:ExMinICW} in context  of the existing literature about both the free (i.e. unconstrained) and the isoperimetric constrained Willmore problems. We will start by discussing how  the minimisers for the free problem already provide partial solutions to Problem~\ref{prob:constrained_Willmore}. 

\begin{problem}[Classical Willmore problem] \label{prob:Willmore}
	Let $g$ be a non-negative integer. Minimise the Willmore energy $\mathcal W$ in the class $\mathcal S_g$. That is, find $f_0 \in \mathcal S_g$ such that $\mathcal W(f_0) \leq \mathcal W(f)$ for any $f \in \mathcal S_g$.
\end{problem}
As a first result on the unconstrained problem, \textsc{Willmore}~\cite{MR0202066} showed that the energy now bearing his name is bounded below by $4\pi$ on the class of closed surfaces, with equality only for round spheres, which solves the genus $g=0$ case. Later, \textsc{Simon}~\cite{MR1243525} proved existence of Willmore minimisers with prescribed genus~$g$, provided 
\begin{equation} \label{eq:intro:compactness_assumption}
	\boldsymbol \beta_g := \inf\{\mathcal W(f): f \in \mathcal S_g\} < \min\{8\pi, \boldsymbol\omega_g\},
\end{equation}
where 
\begin{equation}\label{eq:intro:omega}
\boldsymbol \omega_g = \min\left\{4\pi + \sum_{i = 1}^p(\boldsymbol \beta_{g_i} - 4\pi): g = \sum_{i=1}^pg_i,\,1\leq g_i<g\right\}.
\end{equation}
This assumption is used to obtain compactness in the direct method of calculus of variations. It was already known since the work of \textsc{Willmore} that $\boldsymbol \beta_1 \leq 2\pi^2$ and therefore $\boldsymbol \beta_1 < 8\pi$. Moreover, since by definition $\boldsymbol \omega_1 = \infty$, \textsc{Simon} in particular proved existence of Willmore tori (i.e. genus $g=1$ minimisers). \textsc{Kusner}~\cite{MR996204} showed that $\boldsymbol \beta_g < 8\pi$ for all $g \geq 1$, by estimating the area of minimal surfaces in the 3-sphere $\mathbb S^3$ found by \textsc{Lawson}~\cite{MR0270280}. Hence, in order to prove existence of Willmore minimisers with prescribed genus $g \geq 2$, the missing step was to show that 
\begin{equation} \label{eq:intro:Simons_inequality}
	\boldsymbol \beta_g < \boldsymbol \omega_g.
\end{equation} 
There were some suggestions on that inequality before it was finally proven. Namely \textsc{Simon}~\cite{MR1243525} conjectured that $\boldsymbol \beta_g \geq 6\pi$ for all $g \geq 1$ which would imply $\boldsymbol \omega_g > 8\pi$, reducing the compactness assumption in~\eqref{eq:intro:compactness_assumption} to the $8\pi$-bound proven by \textsc{Kusner}. \textsc{Simon}'s conjecture is now known to be true but we will discuss it  later. Furthermore, he explained that the non-strict inequality
\begin{equation} \label{eq:intro:non-strict_Simon_inequality}
	\boldsymbol \beta_g \leq \boldsymbol \omega_g 
\end{equation}
is indeed true. To see this, he suggested to choose $p$ surfaces $f_1 \in \mathcal S_{g_1},\ldots,f_p\in \mathcal S_{g_p}$ with Willmore energies close to $\boldsymbol \beta_{g_1},\ldots,\boldsymbol \beta_{g_p}$, respectively. Then, to each surface $f_i$, apply a sphere inversion
\begin{equation} \label{eq:intro:sphere_inversion}
	I_{a_i}: \mathbb R^3 \setminus\{a_i\} \to \mathbb R^3, \qquad I_{a_i}(x) = \frac{x-a_i}{|x-a_i|^2}\; ,
\end{equation}
for some point $a_i \in \im f_i$ of multiplicity one, turning the surface $f_i$ into an unbounded surface $I_{a_i}\circ f_i$ with a planar end and Willmore energy $\mathcal W(I_{a_i}\circ f_i) = \mathcal W(f_i) - 4\pi$. (In fact, he suggested to choose $a_i$ close to the image of $f_i$ which results in a surface that already looks like a round sphere; the final construction however will look the same). Then, focus on the part of $I_{a_i}\circ f_i$ that carries the genus of $f_i$, cut away the planar end and glue it into a large round sphere. The glueing can be done at small cost in terms of Willmore energy in such a way that the resulting surface looks like a round sphere with a cap of $g_i$ handles, having the same genus as $f_i$ and Willmore energy close to the sum $\mathcal W(\mathbb S^2) + \mathcal W(I_{a_i}\circ f_i)$. Glueing suitable sphere inversions of the surfaces $f_1,\ldots,f_p$ all into the same large sphere, results in a surface $f$ with $\genus(f) = \genus(f_1) + \ldots + \genus(f_p)$ that looks like a round sphere with $p$ caps and Willmore energy    
\begin{equation*}
	\mathcal W(f) \approx \mathcal W(\mathbb S^2) + \sum_{i = 1}^p\mathcal W(I_{a_i}\circ f_i) \approx 4\pi + \sum_{i = 1}^p(\boldsymbol \beta_{g_i} - 4\pi)
\end{equation*}
which indeed implies the non-strict inequality~\eqref{eq:intro:non-strict_Simon_inequality}. In fact, in order to prove either of the inequalities \eqref{eq:intro:Simons_inequality} or \eqref{eq:intro:non-strict_Simon_inequality} one might assume that $p=2$ in the definition of $\boldsymbol\omega_g$ (see Equation~\eqref{eq:intro:omega}). The general case then follows by induction. \textsc{Kusner}~\cite{MR996204} developed the \emph{conformal connected-sum} $M\#N$ of two given immersed surfaces $M$ and $N$ in $\mathbb R^3$ satisfying
\begin{equation*}
	\mathcal W(M \# N) = \mathcal W(M) + \mathcal W(N) - 4\pi,
\end{equation*}
which also implies the non-strict inequality \eqref{eq:intro:non-strict_Simon_inequality}. This kind of equation can be found in many mathematical concepts. Notable for instance is that the same equation holds true for the Euler characteristic $\chi$ of the connected sum $M\#N$ of two $n$-manifolds $M$ and $N$:
\begin{equation*}
	\chi(M\#N)= \chi(M) + \chi(N) - \chi(\mathbb S^n).
\end{equation*} 
Later, \textsc{Kusner}~\cite{MR1417949} suggested to invert the two surfaces at nonumbilic points after which the planar end of each surface is asymptotic to the graph of a biharmonic function with higher order terms decaying at least as fast as $1/r$. Therefore, one can then weld together two such inverted surfaces along a line in their planar ends and estimate the saved energy in terms of the energy of a biharmonic graph. Inspired from this idea, \textsc{Bauer-Kuwert} \cite{MR1941840} finally found a proof for the strict inequality \eqref{eq:intro:Simons_inequality} and thus completed the existence proof of Problem~\ref{prob:Willmore}. Given two smoothly immersed surfaces $f_i:\Sigma_i \to \mathbb R^3$ with $i=1,2$ neither of which is a round sphere, they constructed an immersed surface $f:\Sigma \to \mathbb R^3$ with topological type of the connected sum $\Sigma_1\#\Sigma_2$ by inverting the first surface $f_1$ at a nonumbilic point and glueing the inverted surface directly into a large copy of the second surface, again at a nonumbilic point. The glueing was done by the graph of a biharmonic function. Thereby they inferred
\begin{equation}\label{eq:intro:BK_inequality}
	\mathcal W(f) < \mathcal W(f_1) + \mathcal W(f_2) - 4\pi
\end{equation}
which implies \eqref{eq:intro:Simons_inequality}. 

An alternative way to prove the strict inequality for the high genus case follows from \textsc{Kuwert-Li-Sch\"atzle} \cite{MR2597505}. They proved that $\lim_{g\to\infty}\boldsymbol\beta_g = 8\pi$ which then implies $\lim_{g\to\infty}\boldsymbol\omega_g>8\pi$. Later, existence of Willmore minimisers was also proven by a  parametric approach in independent works of \textsc{Kuwert-Li} \cite{MR2928715} (building on top of previous work of 
\textsc{M\"uller-\v{S}ver\'{a}k} \cite{MR1366547}) and \textsc{Rivi\`ere} \cite{MR3008339} (building on top of \textsc{H\'elein}'s moving frames technique \cite{MR1913803}). In the parametric approach, the inequality in \eqref{eq:intro:compactness_assumption} is needed to obtain compactness in the moduli space of conformal structures over a Riemann surface.
\medskip

We next discuss the connection between the classical Willmore problem \ref{prob:Willmore} and the  isoperimetric constrained Willmore problem \ref{prob:constrained_Willmore}.
Given any smoothly embedded surface $f:\Sigma \to \mathbb R^3$ (such are Willmore minimisers), one can find a smooth curve $\gamma: (0,\infty) \to \mathbb R^3$ with $\im(\gamma)\cap\im(f) = \emptyset$ such that the sphere inversions $I_\gamma\circ f$ (see the definition in \eqref{eq:intro:sphere_inversion}) have the following properties. The isoperimetric ratio $\iso(I_{\gamma(r)}\circ f)$ varies smoothly in $r$,
\begin{equation*}
	\lim_{r \to 0+} \dist(\gamma(r),\im(f)) = 0,\qquad \lim_{r \to \infty} \dist(\gamma(r),\im(f)) = \infty 
\end{equation*}
and, most importantly,
\begin{equation*}
	\lim_{r \to 0+} \iso(I_{\gamma(r)}\circ f) = \iso(\mathbb S^2),\qquad \lim_{r \to \infty} \iso(I_{\gamma(r)}\circ f) = \iso(f). 
\end{equation*}
Given any integer $g \geq 1$ and any (unconstrained) Willmore minimiser $\Sigma_g \in \mathcal S_g$, it follows that for all isoperimetric ratios $\sigma$ in the non-empty interval
\begin{equation} \label{eq:intro:solution_interval}
	(\iso(\mathbb S^2),\iso(\Sigma_g)],
\end{equation}
the isoperimetric constrained Willmore problem \ref{prob:constrained_Willmore} corresponding to $g$ and $\sigma$ has a smoothly embedded solution, given by a suitable sphere inversion of $\Sigma_g$.  In particular, the function $\beta_g:(\sqrt[3]{36\pi},\infty) \to \mathbb R$  (defined in \eqref{eq:intro:min_iso_constr_energy})  is constant on the interval in \eqref{eq:intro:solution_interval} and non-decreasing on the whole domain. 

It follows by the proof of the Willmore conjecture by \textsc{Marques-Neves}~\cite{MR3152944} that
\begin{equation}\label{eq:intro:MN_inequality}
	\boldsymbol \beta_g \geq 2\pi^2 \qquad \text{for all }g\geq1,
\end{equation}
with equality attained only on the Clifford torus (or inversions of it). In particular, $\boldsymbol \beta_g \geq 6\pi$ as conjectured by \textsc{Simon} and thus $\boldsymbol\omega_g > 8\pi$. This fact reduces the compactness assumption in \eqref{eq:intro:compactness_assumption} to the $8\pi$-bound proven by \textsc{Kusner}. Thereby, \textsc{Marques-Neves} provide another proof for the strict inequality \eqref{eq:intro:Simons_inequality} used to show existence of solutions for the classical Willmore problem~\ref{prob:Willmore}, alternatively to the one of \textsc{Bauer-Kuwert}. Thus the interval in \eqref{eq:intro:solution_interval} for $g=1$ reads as
\begin{equation} \label{eq:intro:solution_interval_torus}
	\biggl(\sqrt[3]{36\pi},\textstyle\sqrt[3]{16\sqrt{2}\pi^2}\biggr].
\end{equation}
Notice that recently, \textsc{Rivi\`ere} \cite{riviere2020minmax} gave a PDE based proof of the Willmore conjecture (i.e. Equation~\eqref{eq:intro:MN_inequality}).

The classical Willmore problem only provides solutions for the isoperimetric constrained Willmore problem via sphere inversions for genus $g \geq 1$, as for $g=0$ the minimiser of the Willmore energy coincides with the minimiser of the isoperimetric ratio giving thus an empty interval in \eqref{eq:intro:solution_interval}. The genus $g=0$ case was fully solved by \textsc{Schygulla} \cite{MR2928137} in the ambient approach (and later generalised to the non-zero spontaneous curvature case by the authors \cite{MR4076069} using the parametric approach). We call \emph{Schygulla spheres} the minimisers of the genus $g=0$ case with isoperimetric ratio $\sigma$,  and denote them with $\mathbb S(\sigma)$. 
The first existence result for the genus $g\geq1$ case of the isoperimetric constrained Willmore problem was given by \textsc{Keller-Mondino-Rivi\`ere} \cite{MR3176354}. They proved existence of smoothly embedded minimisers for all isoperimetric ratios $\sigma$ satisfying 
\begin{equation} \label{eq:intro:compactness_iso_constr}
	\beta_g(\sigma) < \min\{8\pi,\boldsymbol\omega_g,(\boldsymbol \beta_g + \beta_0(\sigma)-4\pi)\}.
\end{equation}
Compare this inequality with \textsc{Simon}'s compactness assumption \eqref{eq:intro:compactness_assumption}. \textsc{Schygulla} also showed that $\beta_0(\sigma) = \mathcal W(\mathbb S(\sigma))$ is continuous in $\sigma$. Hence the right hand side in \eqref{eq:intro:compactness_iso_constr} is continuous in $\sigma$. Consequently, by the result of \textsc{Keller-Mondino-Rivi\`ere}, one can show (see \cite[Theorem 1.4]{MR3176354}) that the set of isoperimetric ratios for which there exist constrained minimisers, is an open interval containing the interval in \eqref{eq:intro:solution_interval}.

By the result of \textsc{Marques-Neves}, the constant $\boldsymbol\omega_g$ in the right hand side of \eqref{eq:intro:compactness_iso_constr} is redundant. The main result of this paper will be that also the last constant can be neglected, reducing the compactness assumption \eqref{eq:intro:compactness_iso_constr} to 
\begin{equation*}
	\beta_g(\sigma) < 8\pi.
\end{equation*} 
It is yet unknown whether or not this inequality is always satisfied. Instead, it is clear that the non-strict inequality
\begin{equation*}
	\beta_g(\sigma) \leq 8\pi
\end{equation*}
holds true. This can be seen by taking two concentric spheres of nearly the same radii and connecting them with $(g+1)$ catenoidal necks, resulting in a surface of genus $g$. This construction was carried out by \textsc{K\"{u}hnel-Pinkall} \cite{MR868618}. The isoperimetric ratio of these surfaces tends to infinity, while the Willmore energy approaches $8\pi$ from above.

 We can now state the main result of the present paper.

\begin{theorem} \label{thm:connected_sum}
Suppose $\Sigma_{1}, \Sigma_{2}$ are two closed surfaces, $f_1: \Sigma_1 \to \mathbb R^3$ is a smooth embedding, $f_2: \Sigma_2 \to \mathbb R^3$ is a smooth immersion, and neither $f_1$ nor $f_2$ parametrise a round sphere. Denote with $\Sigma$ the connected sum $\Sigma_1 \# \Sigma_2$. Then there exists a smooth immersion $f:\Sigma \to \mathbb R^3$  such that
	\begin{equation} \label{eq:intro:iso_constraint} 
		\quad \iso(f) = \iso(f_2)
	\end{equation}
	and
	\begin{equation} \label{eq:intro:strict_inequality}
		\mathcal W(f) < \mathcal W(f_1) + \mathcal W(f_2) - 4\pi. 
	\end{equation}
	 Moreover, if also $f_2$ is an embedding, then $f$ is an embedding as well.
\end{theorem}
Recall that the inequality in \eqref{eq:intro:strict_inequality} has been proven by \textsc{Bauer-Kuwert} \cite{MR1941840} in order to solve the classical Willmore problem (see \eqref{eq:intro:BK_inequality}). It is novel that the same inequality remains valid under the additional condition on the isoperimetric ratios \eqref{eq:intro:iso_constraint}. Indeed, in order to prove Theorem \ref{thm:connected_sum}, we use the same connected sum construction developed by \textsc{Bauer-Kuwert}. It will be shown in Section \ref{sec:Iso_balance} that the connected sum already satisfies Equation \eqref{eq:intro:iso_constraint} asymptotically (see Lemma~\ref{lem:isoperimetric_asymptotics}). We then adjust the isoperimetric ratio by applying a first variation of the surface $f_2$ supported away from the pasting region, inspired by \textsc{Huisken}'s volume preserving mean curvature flow \cite{MR921165}  (see Lemma~\ref{lem:initial_mcf} in Section~\ref{sec:initial_mcf}). Using existence of smoothly embedded Schygulla spheres as well as existence of smoothly embedded Willmore minimisers, we infer the following corollary.

\begin{corollary} \label{cor:KMR-inequality}
	Given any integer $g\geq1$, there holds
	\begin{equation*} \label{eq:intro:KMR-inequality}
		 \beta_g(\sigma) < \boldsymbol \beta_g + \beta_0(\sigma) - 4\pi, \quad \text{ for all $\sigma > \sqrt[3]{36\pi}$.}
	\end{equation*}
\end{corollary}

This corollary answers to a question raised by \textsc{Keller-Mondino-Rivi\`ere}, in \cite[Remark 1.7(iii)]{MR3176354}. Our corollary together with their results leads to the following statement, regarding Problem~\ref{prob:constrained_Willmore}. 

\begin{corollary}\label{cor:ExMinICW}  
	Given any integer $g\geq1$, the function $\beta_g(\cdot)$ (defined as in \eqref{eq:intro:min_iso_constr_energy}) is non-decreasing, continuous, and bounded: $4\pi<\beta_g(\cdot) \leq 8\pi$. Moreover, $\beta_g(\sigma)$ is attained by a smoothly embedded minimiser for all $\sigma$ with $\beta_g(\sigma) < 8\pi$.
\end{corollary}

\begin{remark}
	The Willmore energy of a torus of revolution with radii $0<r<R$ is given by $\mathcal W(\mathbb T_{c}) = \frac{\pi^2}{c\sqrt{1-c^2}}$, where $c=r/R$. Its minimum is attained at $c = 1/\sqrt{2}$ which results in the  Clifford torus. Moreover,
	\begin{equation*}
	c_1:= \inf\{c > 0: \mathcal W(\mathbb T_c) < 8\pi\} = \textstyle\sqrt{\frac{1}{2} - \frac{\sqrt{16 - \pi^2}}{8}}.
	\end{equation*}
	The isoperimetric ratio of a torus of revolution $\mathbb T_c$ can be computed as $\iso(\mathbb T_c) = \sqrt[3]{16 \pi^2/c}$. We thus obtain existence of solutions for all isoperimetric ratios in the interval
	\begin{equation} \label{eq:intro:final_solution_interval}
		\bigl(\iso(\mathbb S^2),\iso(\mathbb T_{c_1})\bigr) = \left(\sqrt[3]{36\pi}, \sqrt[3]{16 \pi^2/\textstyle\sqrt{\frac{1}{2} - \frac{\sqrt{16 - \pi^2}}{8}}}\right).
	\end{equation} 
	Notice that for the Clifford torus $\mathbb T = \mathbb T_{1/\sqrt{2}}$ there holds
	\begin{equation*}
		\iso(\mathbb T) = \sqrt[3]{16\sqrt{2}\pi^2} < \iso(\mathbb T_{c_1}).
	\end{equation*} 
	Hence, our solution interval \eqref{eq:intro:final_solution_interval} is strictly larger than the interval in \eqref{eq:intro:solution_interval_torus}.
\end{remark}

\begin{remark}
As recalled above, \cite{MR3176354} showed that the solution interval is open and contains the interval in \eqref{eq:intro:solution_interval}. As a consequence of our main results, we obtain that an improved upper bound for the solution interval is given by $\inf\{\sigma: \beta_g(\sigma) = 8\pi\}$.
\end{remark}

\begin{remark}
	It is expected that if $\beta_g(\sigma)<8\pi$ for all $\sigma>\sqrt[3]{36\pi}$, then
	\begin{equation*}
		\lim_{\sigma \to \infty}\beta_g(\sigma) = 8\pi.
	\end{equation*}
	Indeed, this was proven for $g=0$ by \textsc{Schygulla} \cite{MR2928137} (see also \cite{MR3842922} for a detailed blow-up analysis).
\end{remark}

\textbf{Acknowledgements.}
A.M. is supported by  the ERC Starting Grant  802689 ``CURVATURE''. C.S. is supported by the EPSRC as part of the MASDOC DTC at the University of Warwick, grant No. EP/HO23364/1. 

\section{A suitable volume preserving  first variation}
\label{sec:initial_mcf}
In this section, we construct a variation of a given surface that initially linearly decreases the isoperimetric ratio, provided the surface is not a round sphere. The variation vector field can even be supported away from a given point. Note that the variation constructed below coincides at first order with the volume preserving mean curvature flow as developed by \textsc{Huisken}~\cite{MR921165},  multiplied by a suitable cut-off function.

\begin{lemma} \label{lem:initial_mcf}
    Suppose $f: \Sigma \to \mathbb R^3$ is a smoothly immersed closed surface which is not a round sphere and $q\in \Sigma$ is any given point. Then, there exists $q \neq p \in \Sigma$ with the following property. 
    
    For each neighbourhood $U$ of $p$ there exists a smooth normal vector field $\xi: \Sigma \to \mathbb R^3$ compactly supported in $U$ such that for $f_t:= f + t\xi$ with $t \in \mathbb R$, the function $t \mapsto \iso(f_t)$ is differentiable at $t = 0$, and 
    \begin{equation*}
    	\left. \frac{d}{\mathrm dt} \right|_{t=0} \iso(f_t) \neq 0.
    \end{equation*}   
    Moreover, 
    \begin{equation*}
    	\mathcal W(f_t) = \mathcal W(f) + O(t) \qquad \text{as }t\to 0. 
    \end{equation*}
\end{lemma}

\begin{proof}
    First of all, recall that for any smooth vector field $\xi: \Sigma \to \mathbb R^3$, the family $f_t = f + t\xi$ defines a variation of the immersion $f$. In particular, for small $t \in \mathbb R$, the map $f_t:\Sigma \to \mathbb R^3$ is again a smooth immersion. Thus area, volume, and Willmore energy are defined for $f_t$ with $t$ small. 
    
    By a classical theorem of \textsc{Alexandrov} \cite{MR0143162}, since $f: \Sigma \to \mathbb R^3$ is not a round sphere, the mean curvature $H$ cannot be constant. Therefore, we can choose a point $p$ in the non-empty boundary of the set
    \begin{equation*}
    	\{x \in \Sigma: H(x) = \max \im H\},
    \end{equation*} 
    where $\im H$ is the image of the mean curvature $H$. In fact, given any $c \in \im H$, we might as well have chosen $p$ in the non-empty boundary of the level set $\{H = c\}$. In particular, we can make sure that $p\neq q$. Now, given any neighbourhood $U$ of $p$, we pick a smooth function $\varphi: \Sigma \to \mathbb R$ compactly supported in $U$ such that $\varphi \geq 0$ and $\varphi(p) = 1$. Let $n:\Sigma \to \mathbb S^2$ be the Gau{\ss} map and define the constant $h$ and the vector field $\xi$ by
    \begin{equation*}
    	h = \int_{\Sigma}\varphi H\,\mathrm d\mu\left/\int_{\Sigma}\varphi\,\mathrm d\mu \right. \quad \text{and} \quad \xi = \varphi(H - h) n. 
    \end{equation*}
    Then, $\xi:\Sigma \to \mathbb R^3$ is a smooth vector field compactly supported in $U$. Using the first variation formula of the volume, we compute for $f_t = f +t\xi$ that
    \begin{equation} \label{eq:imcf:volume}
    	\left. \frac{d}{\mathrm dt} \right|_{t=0} \vol(f_t) = - \int_{\Sigma}n \cdot \xi\,\mathrm d\mu = - \int_{\Sigma}\varphi (H - h)\,\mathrm d\mu = 0.
    \end{equation}
    Moreover, using the first variation formula of the area, it follows
    \begin{equation} \label{eq:imcf:area}
    	\begin{split}
	    	\left. \frac{d}{\mathrm dt} \right|_{t=0} \area(f_t) & = - \int_{\Sigma} H n \cdot \xi\,\mathrm d\mu \\
    		& = -\int_{\Sigma}\varphi H(H - h)\,\mathrm d\mu = -\int_{\Sigma}\varphi(H - h)^2\,\mathrm d\mu<0.
    	\end{split} 
    \end{equation}
    The last expression is non-zero due to our choice of the point $p$ and the function $\varphi$. Using \eqref{eq:imcf:volume} and \eqref{eq:imcf:area}, we infer that
    \begin{equation*}
    	\left. \frac{d}{\mathrm dt} \right|_{t=0} \iso(f_t) = \left.-\int_{\Sigma}\varphi(H - h)^2\,\mathrm d\mu \right/\vol(f)^{\frac{2}{3}} <0.
    \end{equation*}
    Finally, using the first variation formula for the Willmore energy, we see that the function $t\mapsto \mathcal W(f_t)$ is differentiable at $t = 0$ which implies the conclusion.
\end{proof}

\section{Isoperimetric balance of the connected sum} \label{sec:Iso_balance}
In this section we recall the connected sum construction developed by \textsc{Bauer--Kuwert} \cite{MR1941840} and estimate its change of isoperimetric ratio (see Lemma~\ref{lem:isoperimetric_asymptotics}). Then, we prove Theorem~\ref{thm:connected_sum} by applying the volume preserving variation constructed in Lemma \ref{lem:initial_mcf} to the connected sum.

Let $f_i:\Sigma_i \to \mathbb R^3$ for $i=1,2$ be two smoothly immersed closed surfaces neither of which is a round sphere such that 
\begin{equation} \label{eq:ib:surfaces_at_zero}
    f_i^{-1}\{0\} = \{p_i\}, \quad \text{for some } p_i \in \Sigma_i, \qquad \im \mathrm Df_i(p_i) = \mathbb R^2 \times \{0\}.
\end{equation}
For some $\rho > 0$, one can then pick smooth local graph representations 
\begin{equation*}
    f_1(z) = (z,u(z)), \quad f_2(z) = (z, v(z)) \qquad \text{for } z \in D_\rho,
\end{equation*}
where $D_\rho$ is the open disk $\{z \in \mathbb R^2: |z| < \rho\}$. Letting $P,Q: \mathbb R^2 \times \mathbb R^2 \to \mathbb R$ be the second fundamental forms at the origin of $f_1$ and $f_2$, respectively, we define the error terms $\phi$ and $\psi$ such that
\begin{align*}
      &u(z) = p(z) + \varphi(z), \quad \text{where } p(z) = \frac{1}{2} P(z,z) \qquad \text{for } z \in D_\rho,\\
      &v(z) = q(z) + \psi(z), \quad \text{where } q(z) = \frac{1}{2} Q(z,z) \qquad \text{for } z \in D_\rho.
\end{align*}
We denote the trace-free parts of the second fundamental forms with
\begin{equation*}
    P^{\circ}(w,z) = P(w,z) - \frac{(\tr P)}{2} w \cdot z, \qquad Q^{\circ}(w,z) = Q(w,z) - \frac{(\tr Q)}{2} w \cdot z.
\end{equation*}
In view of \cite[Lemma~4.5]{MR1941840}, we may assume that in addition to \eqref{eq:ib:surfaces_at_zero}, there also holds 
\begin{equation} \label{eq:ib:surfaces_oriented}
    \langle P^{\circ}, Q^{\circ} \rangle > 0.
\end{equation}
By \cite[Lemma~2.3]{MR1941840}, the inverted and translated surface
\begin{equation*}
    f_1^{\circ}:\Sigma_1 \setminus \{p_1\} \to \mathbb R^3, \qquad f_1^{\circ}(p) = \frac{f_1(p)}{|f_1(p)|^2} - \frac{(\tr P)}{4}e_3,
\end{equation*}
where $e_3 = (0,0,1)$ is the third unit vector in $\mathbb R^3$, has a graph representation at infinity. That is, outside of a large ball around zero, $f_1^{\circ}$ is given by the graph of a smooth function $u^\circ$ on $\mathbb R^2 \setminus D_R$ for some $R>0$ with
\begin{equation} \label{eq:ib:representation_inversion}
    u^{\circ}(z) = p^{\circ}(z) + \varphi^{\circ}(z), \quad \text{where } p^{\circ}(z) = \frac{1}{2}P^{\circ}\left(\frac{z}{|z|},\frac{z}{|z|}\right)  
\end{equation}
such that the error term satisfies 
\begin{equation} \label{eq:ib:remainder_estimate_inversion}
    |z||\varphi^{\circ}(z)| + |z|^2|\mathrm D \varphi^{\circ}(z)| + |z|^3|\mathrm D^2 \varphi^{\circ}(z)| \leq C \qquad \text{for } z \in \mathbb R^2 \setminus D_R.
\end{equation}
Given any function $w:\Omega \to \mathbb R$ for $\Omega  \subset \mathbb R^2$ and given any scalar $\lambda > 0$, we define the scaled function $w_\lambda$ by
\begin{equation*}
	w_\lambda:\Omega_\lambda = \{z \in \mathbb R^2: \lambda^{-1}z\in \Omega\} \to \mathbb R, \qquad w_\lambda(z) = \lambda w(\lambda^{-1}z).
\end{equation*} 
Hence, for small $\alpha, \beta > 0$, the graph representations of the scaled surfaces $\alpha f_1^{\circ}$ and $(1/\beta)f_2$ are given by
\begin{align*}
    u_\alpha^{\circ}(z) & = p_\alpha^{\circ}(z) + \varphi_\alpha^{\circ}(z) && \text{for } z \in \mathbb R^2 \setminus D_{\alpha R},\\
    v_{1/\beta}(z) & = q_{1/\beta}(z) + \psi_{1/\beta}(z) && \text{for } z \in D_{\rho/\beta}.
\end{align*}
Next, pick a smooth function $\eta: \mathbb R \to \mathbb R$ such that
\begin{equation*}
    \eta(t) = 
    \begin{cases}
        0 & t \leq (1/4)\sqrt{\alpha} \\
        1 & t \geq (3/4)\sqrt{\alpha}
    \end{cases}
\end{equation*}
and such that $|\eta| + \sqrt{\alpha} |\eta'| + \alpha|\eta''| \leq C$ for some $0<C<\infty$ independent of~$\alpha$. Then, for a third parameter $\gamma$ with $0<\alpha, \beta \ll \gamma \ll1$, define for $r = |z|$,
\begin{equation*}
    w(z) = 
    \begin{cases}
        p_{\alpha}^{\circ}(z) + \eta(\gamma - r) \varphi_{\alpha}^{\circ}(z) & \alpha R < r \leq \gamma \\
        q_{1/\beta}(z) + \eta(r - 1) \psi_{1/\beta}(z) & 1 \leq r < \rho/\beta
    \end{cases}
\end{equation*}
and notice that $w = u_\alpha^{\circ}$ for $r \leq \gamma - (3/4) \sqrt{\alpha}$ as well as $w = v_{1/\beta}$ for $r \geq 1 + (3/4)\sqrt{\alpha}$. Moreover, on $D_1 \setminus D_\gamma$, let $w$ be the unique solution of the bi-harmonic Dirichlet--Neumann problem (see Lemma~3.1 and Lemma~3.2 in \cite{MR1941840})
\begin{equation}\label{eq:ib:biharmonic_interpolation}
    \begin{cases}
        \Delta^2w = 0 & \text{in } D_1 \setminus D_\gamma, \\
        w = p_\alpha^{\circ}, \quad \partial_r w = \partial_rp_{\alpha}^{\circ} & \text{on } |z| = \gamma, \\
        w = q_{1/\beta}, \quad \partial_r w = \partial_r q_{1/\beta} & \text{on } |z| = 1. 
    \end{cases}
\end{equation}
To define the pasted surface, let $U$ be the complement in $\Sigma_1$ of the preimage of the set $\{z \in \mathbb R^2: \gamma - \sqrt{\alpha} < |z| < \infty\}$ under the map $\alpha \cdot \pi \circ f_1^{\circ}$, where $\pi: \mathbb R^3 \to \mathbb R^2$ denotes the orthogonal projection. Analogously, let $V$ be the complement in $\Sigma_2$ of the preimage of the set $\{z \in \mathbb R^2: |z| < 1 + \sqrt{\alpha}\}$ under the map $(1/\beta) \cdot \pi \circ f_2$. Moreover, let $W = \{z \in \mathbb R^2:\gamma - \sqrt{\alpha} \leq |z| \leq 1 + \sqrt{\alpha}\}$. 
Then, we can write the connected sum $\Sigma = \Sigma_1 \# \Sigma_2$ as $\Sigma = (U \cup V \cup W)/\sim$, where the identification~$\sim$ is given by
\begin{align*}
    p \sim z & = \alpha \pi(f_1^{\circ}(p)) && \text{for } p \in U, \, z \in W,\\
    q \sim z &= (1/\beta)\pi(f_2(q)) && \text{for } q \in V, \, z \in W.
\end{align*}
Now, the immersion of the patched surface can be defined by
\begin{equation} \label{eq:ib:connected_sum}
    f:\Sigma \to \mathbb R^3, \qquad f(x) = 
    \begin{cases}
        \alpha f_1^{\circ}(p) & x = p\in U \subset \Sigma_1,\\
        (1/\beta) f_2(q) & x = q \in V \subset \Sigma_2,\\
        (z,w(z)) & x = z \in W.
    \end{cases}
\end{equation}
The connected sum satisfies the following energy saving proven by \textsc{Bauer--Kuwert} \cite{MR1941840}.

\begin{lemma}[\protect{See \cite[Lemma~4.4]{MR1941840}}] \label{lem:energy_balance}
    Taking $\beta = t \alpha$ for any $t>0$, and letting $\alpha$ tend to zero, there holds 
    \begin{equation} \label{eq:ib:energy-excess}
    	\begin{split}
        	&\mathcal W(f) - (\mathcal W(f_1) + \mathcal W(f_2) - 4\pi) \\
        	&\qquad = \pi \alpha^2 \Bigl(|P^{\circ}|^2 - t \langle P^{\circ}, Q^{\circ} \rangle + O_t(\gamma^2\log(\gamma)^2) + O_{t,\gamma}(\alpha^{1/2})\Bigr),
        \end{split}
    \end{equation}
    where the constants in $O_t$ and $O_{t,\gamma}$ depend on $t$, respectively $t$ and $\gamma$.
\end{lemma}

We will show that the isoperimetric ratio of the connected sum behaves as follows.

\begin{lemma} \label{lem:isoperimetric_asymptotics}
    Taking $\beta = t \alpha$ for any $t>0$, and letting $\alpha$ tend to zero, there holds
    \begin{equation} \label{eq:ib:convergence_iso}
        \iso(f) = \iso(f_2) + O_{t,\gamma}(\alpha^{2+\frac{1}{2}}),
    \end{equation}
    where the constant in $O_{t,\gamma}$ depends on $t$ and $\gamma$.
\end{lemma}

\begin{proof}
    First, we will compute the area of the surface $f:\Sigma \to \mathbb R^3$. By definition of the connected sum in Equation~\eqref{eq:ib:connected_sum}, we can split the area into
    \begin{equation} \label{eq:ib:area_splitting}
        \area(f) =  \area(f|_{U}) + \area(f|_{W}) + \area(f|_{V}).
    \end{equation}
	Let 
	\begin{equation*}
		U_1 = \Sigma_1 \setminus (\pi \circ f_1^\circ)^{-1}\{z \in \mathbb R^2: R < |z| < \infty\}
	\end{equation*}
	where again, $\pi$ denotes the orthogonal projection of $\mathbb R^3$ onto $\mathbb R^2$. Then, we can write
	\begin{equation} \label{eq:ib:area_inversion}
		\area(f|_{U}) = \alpha^2 \area(f_1^\circ|_{U_1}) + \int_{D_{\gamma - \sqrt{\alpha}}\setminus D_{\alpha R}}\sqrt{1 + |\mathrm Du_\alpha^\circ|^2}\,\mathrm d \mathcal L^2,
	\end{equation}
	where $\mathcal L^2$ denotes the 2-dimensional Lebesgue measure. For $p^\circ$ defined as in Equation~\eqref{eq:ib:representation_inversion}, we have
	\begin{equation*}
	\mathrm{D} p^\circ(z) = P^\circ\left(\frac{z}{|z|}, \mathrm D\left(\frac{z}{|z|}\right)\right), \qquad \mathrm  D\left(\frac{z}{|z|}\right) = \frac{\mathrm{Id}}{|z|} -  \frac{\langle z,\cdot\rangle}{|z|^3}z.
	\end{equation*}
	Hence, $|p^\circ(z)| + |z||\mathrm D p^\circ(z)|\leq C$ for $z \in \mathbb R^2 \setminus D_{R}$ and after scaling,
	\begin{equation*}
		|p_\alpha^\circ(z)| + |z||\mathrm D p_\alpha^\circ(z)| \leq C\alpha \qquad \text{for } z \in \mathbb R^2 \setminus D_{\alpha R}.
	\end{equation*}
	Moreover, from the error estimation in Equation~\eqref{eq:ib:remainder_estimate_inversion},
	\begin{equation*}
		|z||\varphi_\alpha^\circ(z)| + |z|^2|\mathrm D\varphi_\alpha^\circ(z)| \leq C \alpha^2 \qquad \text{for } z \in \mathbb R^2 \setminus D_{\alpha R}.
	\end{equation*}
    Using $u_\alpha^\circ = p_\alpha^\circ + \varphi_\alpha^\circ$ we thus infer
    \begin{align} \label{eq:ib:scaled_inversion1}
      	|u_\alpha^\circ(z)|+|\mathrm Du_\alpha^{\circ}(z)| &\leq C && \qquad \text{for } z \in D_{\sqrt{\alpha} R} \setminus D_{\alpha R},\\ \label{eq:ib:scaled_inversion2}
    	|u_\alpha^\circ(z)|+|\mathrm Du_\alpha^{\circ}(z)| &\leq C\sqrt{\alpha}   && \qquad \text{for } z \in \mathbb R^2 \setminus D_{\sqrt{\alpha} R}.
    \end{align}
    Therefore, the area in Equation~\eqref{eq:ib:area_inversion} can be estimated by
    \begin{align*}
        \area(f|_{U}) \leq C \alpha^2  + C \mathcal L^2 \bigl( D_{R\sqrt{\alpha}} \setminus D_{R\alpha}\bigr) + (1 + C\sqrt{\alpha}) \mathcal L^2 ( D_\gamma ) 
        = \mathcal L^2(D_\gamma) + O(\sqrt{\alpha}) 
    \end{align*}
    as $\alpha \to 0$. On the other hand,
    \begin{equation*}
        \area(f|_{U}) \geq \mathcal L^2(D_{\gamma - \sqrt{\alpha}}) = \mathcal L^2(D_{\gamma}) - O(\sqrt{\alpha}) \qquad \text{as } \alpha \to 0
    \end{equation*}
    and thus
    \begin{equation} \label{eq:ib:area_u}
        \area(f|_{U}) = \mathcal L^2(D_{\gamma}) + O(\sqrt{\alpha}) \qquad \text{as } \alpha \to 0.
    \end{equation}
    From \cite[Equation~(4.13)]{MR1941840} it follows that
    \begin{equation} \label{eq:ib:asymptotics_un-inverted}
        |v_{1/\beta}(z)|+|\mathrm Dv_{1/\beta}(z)| \leq C(t) \alpha \qquad \text{for } z \in D_{1 + \sqrt{\alpha}}
    \end{equation}
    and hence,
    \begin{equation} \label{eq:ib:area_v}
            \area((1/\beta) f_2) - \area(f|_{V}) = \int_{D_{1+\sqrt{\alpha}}}\sqrt{1 + |\mathrm D v_{1/\beta}|^2}\,\mathrm d \mathcal L^2\\
            = \mathcal L^2(D_1) + O_t(\sqrt \alpha)
    \end{equation}
    as $\alpha \to 0$. Because of the homogeneity of $p^\circ$ and $q$ (notice that $p_\alpha^\circ = \alpha p^\circ$, $q_{1/\beta} = \beta q$), the parameters $\alpha$ and $\beta$ enter linearly into the boundary values of $w$ on $D_1\setminus D_\gamma$ and thus linearly into the solution \eqref{eq:ib:biharmonic_interpolation} (see (3.29), (3.30), and (3.35) in \cite{MR1941840}). Therefore, using $\beta = t\alpha$ as well as (4.21), (4.22), and (4.25) in \cite{MR1941840}, we infer 
    \begin{equation} \label{eq:ib:asymptotics_biharmonic_interpolation}
        |w(z)| + |\mathrm Dw(z)| \leq C(t,\gamma) \alpha \qquad \text{for } z\in D_{1+\sqrt{\alpha}}\setminus D_{\gamma - \sqrt{\alpha}}.
    \end{equation}
    It follows 
    \begin{equation} \label{eq:ib:area_w}
        	\area(f|_W) = \int_{ D_{1+\sqrt{\alpha}}\setminus D_{\gamma - \sqrt{\alpha}}}\sqrt{1+|\mathrm Dw|^2}\,\mathrm d\mathcal L^2 
        	= \mathcal L^2(D_1 \setminus D_\gamma) + O_{t,\gamma}(\sqrt \alpha)
    \end{equation} 
    as $\alpha \to 0$. Putting \eqref{eq:ib:area_u}, \eqref{eq:ib:area_v}, and \eqref{eq:ib:area_w} into \eqref{eq:ib:area_splitting}, leads to
    \begin{equation*} 
        \area(f) = \area((1/\beta) f_2) +  O_{t,\gamma}(\sqrt{\alpha}) \qquad \text{as } \alpha \to 0
    \end{equation*}
    and thus
    \begin{equation} \label{eq:ib:area_patched_surface}
    	\area(f) = (t\alpha)^{-2}\bigl(\area(f_2) + O_{t,\gamma}(\alpha^{2+\frac{1}{2}})\bigr) \qquad \text{as } \alpha \to 0.
    \end{equation}
    
    Next, we will estimate the volume of the patched surface $f:\Sigma\to \mathbb R^3$. Using the definition of the volume~\eqref{eq:intro:def_volume}, as well as the formula for the Gauss map of graphical surfaces, we estimate
    \begin{align*}
        &|\vol(f) - \vol ((1/\beta) f_2)| \leq \alpha^3\vol(f_1^\circ|_{U_1}) + \int_{D_{\gamma - \sqrt{\alpha}}\setminus D_{\alpha R}}|z||\mathrm D u_\alpha^\circ| + |u_\alpha^\circ|\,\mathrm d \mathcal{L}^2(z) \\
        &\qquad + \int_{D_{1 + \sqrt{\alpha}}\setminus D_{\gamma - \sqrt{\alpha}}}|z||\mathrm D w| + |w|\,\mathrm d \mathcal{L}^2(z) + \int_{D_{1 + \sqrt{\alpha}}}|z||\mathrm D v_{1/\beta}| + |v_{1/\beta}|\,\mathrm d \mathcal{L}^2(z).
    \end{align*}
    In view of \eqref{eq:ib:scaled_inversion1}, \eqref{eq:ib:scaled_inversion2},  \eqref{eq:ib:asymptotics_un-inverted}, and  \eqref{eq:ib:asymptotics_biharmonic_interpolation}, we can see that the right hand side is uniformly bounded in $\alpha$ for $0<\alpha \ll \gamma \ll1$ and $\beta = t\alpha$. That means,
    \begin{equation*} 
    	\vol(f) = \vol ((1/\beta) f_2) + O_{t,\gamma}(1) \qquad \text{as }\alpha \to 0
    \end{equation*}
    and therefore,
    \begin{equation} \label{eq:ib:volume_patched_surface}
    	\vol(f) = (t\alpha)^{-3}\bigl(\vol (f_2) + O_{t,\gamma}(\alpha^3)\bigr) \qquad \text{as }\alpha \to 0.
    \end{equation}
    Notice that by differentiability of the function $s \mapsto (\vol(f_2) + s)^{-2/3}$ at $s = 0$, there holds
    \begin{equation*}
        \frac{1}{\vol(f_2)^{\frac{2}{3}}} = \frac{1}{(\vol(f_2) + s)^{\frac{2}{3}}} +O(s) \qquad \text{as } s \to 0.
    \end{equation*}
    Thus, using $\beta = t\alpha$,  \eqref{eq:ib:area_patched_surface}, and \eqref{eq:ib:volume_patched_surface}, we infer
    \begin{align*}
        \iso(f) = \frac{\area(f_2)}{(\vol(f_2) + O_{t,\gamma}(\alpha^3))^{\frac{2}{3}}} + \frac{ O_{t,\gamma}(\alpha^{2+\frac{1}{2}})}{(\vol(f_2) + O_{t,\gamma}(\alpha^3))^{\frac{2}{3}}}
        = \iso(f_2) + O_{t,\gamma}(\alpha^{2+\frac{1}{2}})
    \end{align*}
    as $\alpha \to 0$, which finishes the proof.
\end{proof}

Now, we can prove Theorem~\ref{thm:connected_sum}.

\begin{proof}[Proof of Theorem~\ref{thm:connected_sum}]
	
	Let $\Sigma_1,\Sigma_2$ be two closed surfaces, $f_1:\Sigma_1 \to \mathbb R^3$ be a smooth embedding, and  $f_2:\Sigma_2 \to \mathbb R^3$ be a smooth immersion such that neither $f_1$ nor $f_2$ parametrise a round sphere. Notice that the multiplicity of $f_2$ does not affect the construction of the connected sum. First, pick $p_i \in \Sigma_i$ for $i=1,2$ according to \eqref{eq:ib:surfaces_at_zero} and \eqref{eq:ib:surfaces_oriented},  with $p_2 \in f_2^{-1}\{0\}$ instead of $\{p_2\} = f_2^{-1}\{0\}$. 
	\\Apply Lemma~\ref{lem:initial_mcf} to the surface $f_2:\Sigma_2 \to \mathbb R^3$: denote $f_{2,s} = f_2 + s\xi$ with $\xi$ compactly supported away from the point $p_2$ and 
	\begin{align} \label{eq:ib:mcf_Willmore}
		\mathcal W(f_{2,s}) &= \mathcal W(f_2) + O(s) && \text{as } s\to 0, \\ \label{eq:ib:mcf_iso}
		\iso(f_{2,s}) & = \iso(f_2) - c_2s + o(s) && \text{as } s\to 0,
	\end{align} 
	for some $c_2>0$. Now, apply the connected sum construction described in this section to the surfaces $f_1:\Sigma_1 \to \mathbb R^3$ and $f_{2,s}:\Sigma_2 \to \mathbb R^3$; in this way we obtain the glued surface $f_{s,\alpha}:\Sigma\to \mathbb R^3$, where $\Sigma$ is the connected sum of $\Sigma_1$ and $\Sigma_2$. Notice that the right hand side in Equation~\eqref{eq:ib:energy-excess} does not depend on $s$ as the vector field $\xi$ is compactly supported away from the patching area. Therefore, we can first choose $t>0$ large enough such that $|P^\circ|^2 - t\langle P^\circ, Q^\circ \rangle <0$ and then choose $0<\gamma <1$ small enough such that still $|P^\circ|^2 - t\langle P^\circ, Q^\circ \rangle + O_{t}(\gamma^2 \log(\gamma)^2) <0$ to obtain from Lemma~\ref{lem:energy_balance} that
	\begin{equation}\label{eq:ib:applied_energy_balance}
		\mathcal W(f_{s,\alpha}) - (\mathcal W(f_1) + \mathcal W(f_{2,s}) - 4\pi) = -c\alpha^2 +O(\alpha^{2+\frac{1}{2}})
	\end{equation}  
	as $\alpha \to 0$ for some $c>0$. Putting \eqref{eq:ib:mcf_Willmore} into \eqref{eq:ib:applied_energy_balance} and \eqref{eq:ib:mcf_iso} into \eqref{eq:ib:convergence_iso}, we infer
	\begin{align} \label{eq:ib:final_Willmore}
		\mathcal W(f_{s,\alpha}) - (\mathcal W(f_1) + \mathcal W(f_2) - 4\pi) &= -c\alpha^2 +O(\alpha^{2+\frac{1}{2}}) + O(s) \\ \label{eq:ib:final_iso}
		\iso(f_{s,\alpha}) - \iso(f_2) &= O(\alpha^{2+\frac{1}{2}}) - c_2s + o(s)
	\end{align}
	as $s, \alpha \to 0$. Picking any $2 < m < 2 + \frac{1}{2}$, we see that for small $\alpha>0$ and for $|s| \leq \alpha^m$, the right hand side in \eqref{eq:ib:final_Willmore} is strictly negative, while for $s = \alpha^m$ the right hand side in \eqref{eq:ib:final_iso} is strictly negative and for $s = - \alpha^m$ the right hand side in \eqref{eq:ib:final_iso} is strictly positive. Notice that once $\alpha$ is fixed, $\iso(f_{s,\alpha})$ depends continuously on $s$. Therefore, there exists $\alpha >0$ small and $-\alpha^m < s < \alpha^m$ such that the right hand side in \eqref{eq:ib:final_Willmore} is strictly negative, while the right hand side in \eqref{eq:ib:final_iso} is zero. In other words, $f_{s,\alpha}$ satisfies \eqref{eq:intro:iso_constraint} and \eqref{eq:intro:strict_inequality}. Notice that the immersion $f_{s,\alpha}$ is smooth everywhere except on the  boundary of $D_1\setminus D_\gamma$, where the bi-harmonic function meets the  second fundamental forms with Dirichlet and Neumann conditions, see \eqref{eq:ib:biharmonic_interpolation}. In general, $f_{s,\alpha}$ is only $C^{1,1}$-regular.It remains to show that one can approximate $f_{s,\alpha}$ by a smooth immersion without loosing the conditions \eqref{eq:intro:iso_constraint} and \eqref{eq:intro:strict_inequality}. In view of its construction, we can choose a local graph representation of $f_{s,\alpha}$ given by a function $u$ defined on an open subset of $\mathbb R^2$ that contains the boundary of $D_1\setminus D_\gamma$. By multiplying with a cut-off function, one can write $u = u_{\mathrm s} +  u_{\mathrm r}$ such that $u_{\mathrm s}$ is smooth and $u_{\mathrm r}$ is $C^{1,1}$-regular as well as compactly supported. The standard mollification $u_{\mathrm r}^\varepsilon$ of $u_{\mathrm r}$ is smooth, compactly supported, and converges to $u_{\mathrm r}$ as $\varepsilon \to 0$ in the Sobolev space $W^{2,p}$ for all $1\leq p<\infty$. The immersions $f^\varepsilon$ corresponding to $u^\varepsilon:= u_{\mathrm s} + u_{\mathrm r}^\varepsilon$ are smooth and differ from $f_{s,\alpha}$ only on a small neighbourhood of the boundary of $D_1\setminus D_\gamma$. Moreover, there holds
	\begin{equation*}
		|\mathcal W(f^\varepsilon) - \mathcal W(f_{s,\alpha})| + |\iso(f^\varepsilon) - \iso(f_{s,\alpha})| \to 0\qquad \text{as $\varepsilon \to0$.}
	\end{equation*}	
	Hence there exists $\eta>0$  such that, for $\varepsilon>0$ small enough, $f^\varepsilon$ satisfies the following quantified version of  \eqref{eq:intro:strict_inequality}: 
	$$
	\mathcal W(f^\varepsilon)< \mathcal W(f_1) +   \mathcal W(f_2) -4 \pi - \eta.
	$$
	Finally, we  once again apply Lemma \ref{lem:initial_mcf} away from the support of $u^\varepsilon_{\mathrm r}$ to re-establish \eqref{eq:intro:iso_constraint} still keeping the validity of  \eqref{eq:intro:strict_inequality} .

\end{proof}

{
\footnotesize
\bibliographystyle{alpha}
\bibliography{mybib}
}
\end{document}